\documentclass[12pt,a4paper]{amsart}

\usepackage{amssymb}
\usepackage{amsthm}
\usepackage[all]{xy}
\usepackage{graphicx}

\usepackage{newtxtext}
\usepackage{newtxmath}

\usepackage{geometry}

\newcommand{\OO}{\mathcal O}
\newcommand{\EE}{\mathcal E}

\newcommand{\Z}{\mathbb Z}

\newcommand{\N}{\mathbb N}

\newcommand{\Hom}{\operatorname{Hom}}
\newcommand{\End}{\operatorname{End}}

\newcommand{\GL}{\operatorname{GL}}

\newcommand{\id}{\operatorname{id}}

\newcommand{\Ext}{{\operatorname{Ext}}}

\newcommand{\Aut}{\operatorname{Aut}}
\newcommand{\Out}{\operatorname{Out}}
\newcommand{\Outcent}{\operatorname{Outcent}}
\newcommand{\HHH}{\operatorname{HH}}

\newcommand{\rank}{\operatorname{rank}}

\newcommand{\Pic}{{\operatorname{Pic}}}

\newtheorem{defi}{Definition}[section]

\newtheorem{thm}[defi]{Theorem}
\newtheorem{lemma}[defi]{Lemma}
\newtheorem{corollary}[defi]{Corollary}
\newtheorem{notation}[defi]{Notation}
\newtheorem{prop}[defi]{Proposition}

\usepackage{lipsum}

\usepackage[hidelinks]{hyperref}

\title{On the geometry of lattices and finiteness of Picard groups}

\author{Florian Eisele}
\address{Florian Eisele, School of Mathematics and Statistics,
	University of Glasgow,
	University Place,
	Glasgow G12 8QQ,
	United Kingdom
}
\email{florian.eisele@glasgow.ac.uk}
\urladdr{https://feisele.github.io}

\newtheoremstyle{named}{}{}{\itshape}{}{\bfseries}{.}{.5em}{\thmnote{#3}}
\theoremstyle{named}
\newtheorem*{namedtheorem}{Theorem}

\renewcommand{\leq}{\leqslant}
\renewcommand{\geq}{\geqslant}


\begin{document}
	\begin{abstract}
		Let $(K,\OO, k)$ be a $p$-modular system with $k$ algebraically closed and $\OO$ unramified, and let $\Lambda$ be an $\OO$-order in a separable $K$-algebra. We call a $\Lambda$-lattice $L$ rigid if $\Ext^1_{\Lambda}(L,L)=0$, in analogy with the definition of rigid modules over a finite-dimensional algebra. 
		By partitioning the $\Lambda$-lattices of a given dimension into ``varieties of lattices'', we show that 
		there are only finitely many rigid $\Lambda$-lattices $L$ of any given dimension. As a consequence we show that if 
		the first Hochschild cohomology of $\Lambda$ vanishes, then the Picard group and the outer automorphism group of $\Lambda$ are finite. In particular the Picard groups of blocks of finite groups defined over $\OO$ are always finite.
	\end{abstract}

	\maketitle
	
		\vspace{-0.7cm}
	
	\section{Introduction}

	Let $k=\bar k$ be an algebraically closed field of characteristic $p>0$. Let 
	$\OO=W(k)$ be the ring of Witt vectors over $k$, and denote by $K$ the field of fractions of $\OO$. In the representation theory of a finite group $G$ over either of these rings, permutation modules, $p$-permutation modules and endo-permutation modules play a pivotal role. However, to even define permutation modules, one needs to know a group basis of the group ring, a piece of information which is lost when passing to isomorphic or Morita equivalent algebras. Recovering the information lost when forgetting the group basis, or at least quantifying the loss, is a fundamental problem in modular representation theory. This is, for instance, the problem once faces when trying to bridge the gap between Donovan's and Puig's respective conjectures. There are scant results in this direction, apart from Weiss' seminal theorem \cite{WeissRigidity}, which gives a criterion for a lattice to be $p$-permutation requiring only limited knowledge of a group basis (and in some cases none at all). In the present article we study the following property of lattices over an $\OO$-order $\Lambda$ such that $A=K\otimes_\OO \Lambda$ is a separable $K$-algebra, which of course includes lattices over a finite group algebra $\OO G$:
	\begin{defi}
		A $\Lambda$-lattice $L$ is called \emph{rigid} if $\Ext^1_\Lambda(L,L)=0$. 
	\end{defi}
	First and foremost we should point out that permutation lattices over a finite group algebra $\OO G$ are rigid in this sense. The notion of rigid modules over finite-dimensional {$k$-algebras} is widely known and well-studied (see for example \cite{DonaldDeformations, DadeAlgRigid}), but unfortunately permutation $kG$-modules and their ilk usually do not have that property. To see why permutation \emph{lattices} do, we can use a well-known alternative characterisation of rigidity in terms of endomorphism rings, which works for arbitrary $\Lambda$-lattices $L$. Consider the following long exact sequence obtained by applying $\Hom_\Lambda(L,-)$ to the short exact sequence ${0\rightarrow L\rightarrow L \rightarrow L/pL\rightarrow 0}$:
	\begin{equation}
		0 \longrightarrow \End_{\Lambda}(L) \stackrel{p}{\longrightarrow} \End_{\Lambda}(L)\longrightarrow\End_{\Lambda}(L/pL) \longrightarrow \Ext^1_\Lambda(L,L) \stackrel{p}{\longrightarrow} \Ext^1_\Lambda(L,L).
	\end{equation}
	As $A$ is separable, we must have $K\otimes_\OO \Ext^1_{\Lambda}(L,L)\cong \Ext^1_{A} (K\otimes_\OO L, K\otimes_\OO L)=0$, meaning that $\Ext^1_{\Lambda}(L,L)$ is a finitely generated torsion $\OO$-module. By the Nakayama lemma $\Ext^1_{\Lambda}(L,L)$ is zero if and only if multiplication by $p$ is injective on it, which, by exactness, happens if and only if the reduction map 
	\begin{equation}
	\End_\Lambda(L)\longrightarrow \End_{\Lambda}(L/pL)
	\end{equation}
	is surjective. 
	  Now if $R[H\backslash G]$ is a permutation module over an arbitrary commutative ring $R$, then its endomorphisms can be identified with sums $\sum_{h\in HgH} h$ over double cosets $HgH \in H\backslash G /H$. That clearly implies the (well-known) fact that the reduction map $\End_{\OO G}(\OO[H\backslash G]) \longrightarrow \End_{kG}(k[H\backslash G])$ is surjective, which by the above implies that $\OO[H \backslash G]$ is rigid. 
	
	After this short digression on permutation lattices let us state our first main result, which is that rigid lattices enjoy the same discreteness property as rigid modules over finite-dimensional algebras.
	\begin{namedtheorem}[Theorem A]\label{thm rigid main}
		For every $n\in \N$ there are at most finitely many isomorphism classes of rigid $\Lambda$-lattices of $\OO$-rank at most $n$.
	\end{namedtheorem}
 An easy corollary of this is that only finitely many isomorphism classes of $\OO G$-lattices of any given character can be images of permutation lattices under Morita or stable equivalences originating from another group or block algebra.   
	Perhaps unsurprisingly, this result is proved using geometric methods. The basic idea of using a variety of modules (or a variety of complexes, as the case may be) has been successfully applied to modules and complexes over finite-dimensional algebras in many different contexts \cite{DadeAlgRigid, DadeGreenCorrespondent, DonaldDeformations, ZimmermannGeometryChainComplexes, RickardRigidity, RouquierAutomorphFrench}. The idea is typically that one gets a homomorphism from the tangent space of such a variety in a point $M$ into $\Ext^1(M,M)$, whose kernel is the tangent space of the subvariety of points $N$~isomorphic~to~$M$. 
	
	The way we obtain a variety parametrising all lattices with a given $K$-span is quite different from how one proceeds for finite-dimensional $k$-algebras. We essentially start with a fixed lattice, and conjugate an affording representation by a generic matrix. That is how we obtain our analogue of a variety of modules, a \emph{smooth family of $\Lambda$-lattices} (see Definition~\ref{defi smooth family}). 
	  Despite working over $\OO$, these smooth families of lattices are actually parametrised by varieties over $k$. 
	  This is possible due to the theory of Witt vectors. The two main ingredients of Theorem~A are the fact that each  smooth family of lattices contains at most one rigid lattice, up to isomorphism (see Theorem~\ref{thm smooth family open subset}), and the fact that lattices in a given finite-dimensional $A$-module can be appropriately partitioned into finitely many smooth families (see Theorem~\ref{thm lattice alg}). 
	
	Our second main result is an immediate consequence of Theorem~A, but we state it as a theorem nevertheless, since it was the main motivation for writing this paper. The study of Picard groups of blocks of finite group algebras over $\OO$ was initiated in \cite{BKL}. While \cite{BKL} primarily studies the group of Morita self-equivalences induced by bimodules of endo-permutation source, this is where the possibility of Picard groups of blocks always being finite was first raised. The theorem we obtain is much more general, and provides further evidence for the speculation put forward in \cite{EiselePicard} that the Lie algebra of $\Pic_{\OO}(\Lambda)$ (which is shown to be an algebraic group in that paper) should be related to the first Hochschild cohomology of $\Lambda$ in some way or form.
	
	\begin{namedtheorem}[Theorem B]\label{thm pic trivial}
		Assume that $\HHH^1(\Lambda)=0$. Then $\Pic_{\OO}(\Lambda)$ is a finite group. 
	\end{namedtheorem}
	Here $\Pic_{\OO}(\Lambda)$ denotes the group of $\OO$-linear self-equivalences of the module category~of~$\Lambda$. This group is commensurable with the outer automorphism group $\Out_\OO(\Lambda)$, as well as $\Outcent(\Lambda)$ and $\operatorname{Picent}(\Lambda)$, and we could just as well have stated Theorem~B with $\Pic_{\OO}(\Lambda)$ replaced by any of these groups.
	The theorem of course immediately implies the following:
	\begin{corollary}\label{cor pic finite}
		Let $G$ be a finite group, and let $\OO Gb$ be a block. Then $\Pic_\OO(\OO Gb)$ is finite. 
	\end{corollary}
	There is also a more elementary formulation of this fact which is reminiscent of the second Zassenhaus conjecture.
	\begin{corollary}
		For a finite group $G$ there are only finitely many $\mathcal U(\OO G)$-conjugacy classes of group bases of $\OO G$.
	\end{corollary}

	Corollary~\ref{cor pic finite} shows that block algebras over $\OO$ are in some sense extremely rigid, and it turns their Picard groups into a very interesting finite invariant. This is particularly important since K\"ulshammer \cite{KuelshammerDonovan} showed that Picard groups play an important role in both Donovan's conjecture and the classification of blocks of a given defect group. This theme is explored in \cite{EatonDefect16, EEL}. In the way of actual computations, \cite{NebeHertweck} determines the Picard group of the principal block of $A_6$ for $p=3$, \cite{EatonLiveseyPicard} determines Picard groups of almost all blocks of abelian $2$-defect of rank three, and this will certainly not mark the end of the story. Unfortunately, we have very little to offer to aid the calculation of Picard groups. Still, the following might be useful. Note that $\mathcal{T}(\OO Gb)$ denotes the subgroup of $\Pic_{\OO}(\OO Gb)$ of equivalences induced by $p$-permutation bimodules, which is determined in \cite{BKL}. 

	\begin{prop}\label{prop det pic}
		Let $G$ be a finite group, let $\OO G b$ be a block and let $P\leq G$ be one of its defect groups. 
		\begin{enumerate}
		\item An element of $\Pic_{\OO}(\OO G b)$ lies in $\mathcal T(\OO Gb)$ if and only if it sends the $p$-permutation lattice $\OO[P\backslash G]\cdot b$ to a $p$-permutation lattice.
		\item If $\OO[P\backslash G]\cdot b$ is, up to isomorphism, the only rigid lattice in $K[P \backslash G]\cdot b$ with endomorphism ring isomorphic to $\End_{\OO G}(\OO[P\backslash G]\cdot b)$, then 
		$\operatorname{Picent}(\OO Gb)\subseteq \mathcal T(\OO Gb)$.
		\end{enumerate}
	\end{prop}
	The same is true if we set $b=1$ and let $P$ be a Sylow $p$-subgroup of $G$. It is also sufficient to prove that each indecomposable summand $L$ of $\OO[P\backslash G]\cdot b$ is the unique rigid lattice in $K\otimes_\OO L$ with endomorphism ring $\End_{\OO Gb}(L)$. It is however unclear whether one can show such uniqueness in any interesting examples, even though the construction explained in \S\ref{section lattice var} is in principle constructive. Nevertheless,
	Proposition~\ref{prop det pic} highlights the importance of understanding and classifying rigid lattices with a given character. 
	
	One last thing to note is that while we focus on applications to block algebras in this article, there are other types of $\OO$-orders with vanishing first Hochschild cohomology to which Theorem~B applies. Iwahori-Hecke algebras defined over $\OO$, for instance, should have this property by \cite[Theorem 5.2]{GeckRouquier}.

	\begin{namedtheorem}[Notation and conventions] By $\nu_p:\ K \longrightarrow \Z\cup\{\infty\}$ we denote the $p$-adic valuation 
		on $K$. Modules are right modules by default. All varieties are reduced, and by a ``point'' we mean a closed point. $\Lambda$ will always denote an $\OO$-order in a separable $K$-algebra. We will assume that the reader is familiar with the theory of Witt vectors as laid out in \cite[\S5--\S6]{SerreLocalFields}.
	\end{namedtheorem}

	\section{Prerequisites}
	
	Recall that a commutative ring $R$ is called a \emph{strict $p$-ring} if $R$ is complete and Hausdorff with respect to the topology induced by the filtration $R\supseteq pR \supseteq p^2 R \supseteq \ldots$, the element $p$ is not a zero-divisor in $R$ and $\bar R=R/pR$ is a perfect ring of characteristic $p$. The condition that $R$ be complete and Hausdorff is equivalent to the natural homomorphism $R \longrightarrow \varprojlim R/p^i R$ being an isomorphism. The theory of Witt vectors shows (see \cite[\S5 Proposition~10 and \S6 Theorem~8]{SerreLocalFields}) that a strict $p$-ring $R$ with residue ring $\bar R$ is unique up to unique isomorphism (or, to be more precise, the pair $(R, R/pR\rightarrow\bar R)$ is). In particular, there is a unique ring homomorphism $\varphi:\ R \longrightarrow W(\bar R)$, where $W(\bar R)$ denotes the ring of Witt vectors over $\bar R$, making the diagram
	\begin{equation}
		\xymatrix{
			R \ar@{->>}[dr]\ar[r]^\varphi&  W(\bar R)\ar@{->>}[d]\\
			&\bar R
		}
	\end{equation}
	commute, the arrows going down being the natural homomorphism from $R$ into $\bar R$ and the projection onto the first Witt vector component, respectively.  
	
	\begin{notation}
	Let $R$ be a strict $p$-ring with perfect residue ring $\bar R$.
	For $l\in \N$ we let 
	\begin{equation}
	\rho_l:\ R \longrightarrow \bar R^l
	\end{equation}
	denote the map which sends $r\in R$ to the first $l$ components of the Witt vector $\varphi(r)\in W(\bar R)$.
	We say that $r$ \emph{reduces to} $\rho_l(r)$.
	\end{notation}
	We will extend the map $\rho_l$ to vectors and matrices over $R$. There is also no need to explicitly record the ring $R$ in our notation, due to the uniqueness of the isomorphism between a strict $p$-ring and the corresponding ring of Witt vectors explained above. 
	
	A strict $p$-ring with a perfect residue \emph{field} is a complete discrete valuation ring, and if this residue field is moreover algebraically closed we may view any set of truncated Witt vectors over it as affine space. This applies to our ring $\OO$. It is clear that if $f\in \OO[X_1,\ldots,X_n]$ is a polynomial and $r\in \N$, then the condition ``$\nu_p(f(\widehat x_1,\ldots,\widehat x_n))\geq r$'' for $\widehat x_1,\ldots,\widehat x_n\in \OO^n$ is equivalent to certain polynomials in the entries of $\rho_r(\widehat x_1),\ldots,\rho_r(\widehat x_n)$  vanishing. That is, ``$\nu_p(f(\widehat x_1,\ldots,\widehat x_n))\geq r$'' is essentially a ``closed condition'' on Witt vectors. The point of the remainder of this section is to understand the implications of this simple observation.
	
	\begin{defi}\label{defi generic val}
		Let $f\in \OO[X_1, \ldots, X_n]$ be a polynomial, and let $l\in\N$. 
		\begin{enumerate}
		\item We define $\nu_p(f)$ as the minimal $p$-valuation of a coefficient of $f$.
		\item For a point $\mathbf x = (x_{1,0},\ldots,x_{1,l-1},\ldots,x_{n,0},\ldots,x_{n,l-1})\in \mathbb A^{n\cdot l}(k)$ we define 
		 \begin{equation}
		 	\nu_{p, \mathbf x}(f)=\min \{ \nu_p(f(\widehat x_1,\ldots, \widehat x_n)) \ |\ \widehat{\mathbf x}=(\widehat x_1,\ldots, \widehat x_n) \in \OO^n \textrm{ such that } \rho_l(\mathbf{\widehat x})=\mathbf x \}
		 \end{equation}
		 	We call $\nu_{p, \mathbf x}(f)$ the \emph{generic valuation} of $f$ at $\mathbf{x}$.
		 \end{enumerate}
	\end{defi}
 We may extend this generic valuation to $K(X_1\ldots,X_n)$ in the obvious way. One should note though that if $f$ is a rational function rather than a polynomial, then $\nu_p(f(\widehat{\mathbf x}))$ could be either bigger or smaller than $\nu_{p, \mathbf x}(f)$, depending on the choice of $\widehat{\mathbf x}\in \OO^n$ reducing to $\mathbf x$. 

	\begin{prop}\label{prop generic val first properties} Assume the situation of Definition~\ref{defi generic val}.
	\begin{enumerate}
		\item For any $\widehat{\mathbf x}=(\widehat x_1, \ldots, \widehat x_n)\in \OO^n$ reducing to $\mathbf x$
		we have
		\begin{equation}
		\nu_{p,\mathbf x}(f) = \nu_p(f(\widehat x_1 + p^l\cdot Z_1, \ldots, \widehat x_n+p^l\cdot Z_n))
		\end{equation}
		where $Z_1,\ldots, Z_n$ are indeterminates. 
	\item If an $\widehat{\mathbf x} \in \OO^n$  reduces to ${\mathbf x}$, then there is a $\widehat{\mathbf z}\in W(\bar{\mathbb F}_p)^n$ such that 
	\begin{equation}
	\nu_p(f(\widehat {\mathbf x} + p^l \cdot \widehat {\mathbf z}))=\nu_{p, \mathbf x}(f).
	\end{equation}
	\end{enumerate}
	\end{prop}
	\begin{proof}
		\begin{enumerate}
		\item Note that the right hand side cannot be bigger than the left hand side. On the other hand, 
		if $\nu_p(f(\widehat x_1 + p^l Z_1, \ldots, \widehat x_n+p^lZ_n))=m\in \Z_{\geq 0}$, then 
		$p^{-m}f(\widehat x_1 + p^l Z_1, \ldots, \widehat x_n+p^lZ_n)$ reduces to a non-zero polynomial in 
		$k[Z_1,\ldots, Z_n]$, and we can certainly find values for the $Z_i$ for which the polynomial does not vanish. This shows that the right hand side cannot be smaller than the left hand side either.
	\item By the first part we know that $p^{-\nu_{p, \mathbf x}(f)}f(\widehat x_1 + p^l Z_1, \ldots, \widehat x_n+p^lZ_n)$ reduces to a non-zero polynomial $g\in k[Z_1,\ldots, Z_n]$, and 
$\nu_p(f(\widehat {\mathbf x} + p^l \cdot \widehat {\mathbf z}))=\nu_{p, \mathbf x}(f)$ if and only if $g(\mathbf z)\neq 0$, where $\mathbf z$ is the reduction of $\widehat{\mathbf z}$ modulo $p$. Hence, what we are looking for is a $\mathbf z \in \bar {\mathbb F}_p^n$ which avoids the vanishing set of a polynomial defined over $k$. Since $\mathbb A^n(\bar {\mathbb F}_p)$ is Zariski-dense in $\mathbb A^n(k)$ such a $\mathbf z$ exists trivially. \qedhere
		\end{enumerate}
	\end{proof}

	\begin{corollary}\label{cor gen val ext}
	Let $\EE$ be the ring of Witt vectors over an algebraically closed field $k'$ containing $k$,
	and let $f \in \OO[X_1,\ldots,X_n]$ be a polynomial. For $\mathbf x \in \mathbb A^{n\cdot l}(k)$ the value of 
	$\nu_{p, \mathbf x}(f)$ is independent of whether we consider $f$ as a polynomial over $\OO$ or over $\EE$.
	\end{corollary}
	\begin{proof}
		This follows from the first part of Proposition~\ref{prop generic val first properties}.
	\end{proof}

	\begin{corollary}
	In the situation of Proposition~\ref{prop generic val first properties} let $f_1,\ldots,f_d\in \OO[X_1,\ldots, X_n]$ ($d\in \N$) be non-zero polynomials. Then, for any $\widehat{\mathbf x}\in \OO^n$ reducing to $\mathbf x$, we have \begin{equation}
		\nu_{p,\mathbf x}(f_1\cdots f_d)=\nu_p(f_1(\widehat{\mathbf x})\cdots f_d(\widehat{\mathbf x}))
	\end{equation}
	if and only if 
	\begin{equation}
	 \nu_{p,\mathbf x}(f_i)=\nu_p(f_i(\widehat{\mathbf x})) \textrm{ for all $1\leq i\leq d$}.
	\end{equation}
	\end{corollary}
	\begin{proof}
		By the first part of Proposition~\ref{prop generic val first properties} we have 
		$\nu_{p,\mathbf x}(f_1\cdots f_d) = \sum_i \nu_{p, \mathbf x}(f_i)$. Moreover, the valuation of a polynomial at an $\widehat{\mathbf x}$ is always greater than or equal to the generic valuation of the polynomial at $\mathbf x$, but never smaller. The claim follows.
	\end{proof}

	The above shows that if we have finitely many non-zero polynomials $f_1,\ldots,f_d$ ($d\in \N$), then there is always an $\widehat{\mathbf x}$ reducing to $\mathbf x$ such that $\nu_{p}(f_i(\widehat{\mathbf x}))=\nu_{p, \mathbf x}(f_i)$ for all $i$ at once. One last thing we need to understand is what the set of all points $\mathbf x \in \mathbb A^{n\cdot l}(k)$ with $\nu_{p,\mathbf x}(f)\geq r$ (for some given polynomial $f$ and $r\in \Z_{\geq 0}$) looks like geometrically. 
	 Proposition~\ref{prop gen val zariski closed} below answers this, and is essentially a consequence of the fact that such a set is the complement of the image of an open set under the projection onto the first $l$ components of a Witt vector, which will be a closed set. 
	\begin{prop}\label{prop gen val zariski closed}
		Given $f \in \OO[X_1,\ldots,X_n]$, $l\in \N$ and $r\in \Z_{\geq 0}$, there is a closed subvariety ${\mathcal V_{l,\nu_{p,-}(f)\geq r}\subseteq \mathbb A^{n\cdot l}}$ defined over $k$ such that for any algebraically closed $k'\supseteq k$ we have
		\begin{equation}
		\mathcal V_{l,\nu_{p,-}(f)\geq r}(k')=\{ \mathbf x \in \mathbb A^{n\cdot l}(k') \ | \ \nu_{p,\mathbf x}(f) \geq r \}
		\end{equation}
	\end{prop}
	\begin{proof}
		Fix an algebraically closed $k'\supseteq k$, and set $\EE=W(k')$. 
		Let
		 $\mathbf x \in \mathbb A^{n\cdot l}(k')$, and let $\widehat{\mathbf x}\in \EE^n$ be the element reducing to it such that the first $l$ components of the Witt vector $\widehat x_i$ are given by the 
		 $x_{i,j}$ for $0\leq j \leq l-1$, and all other components are zero (note that we use two indices to refer to the $n\cdot l$ entries of $\mathbf x$, which is more natural since we want to think of $\mathbf x$ as an element of $k^{n\times l}$, rather than $k^{n\cdot l}$).
		As $\EE$ is the ring of Witt vectors over $k'$ and $f$ is defined over $k$, we get polynomials ${f_i\in k[X_{1,0},\ldots, X_{1,l-1}, \ldots, X_{n,0},\ldots, X_{n,l-1}, Z_{1,j},\ldots, Z_{n,j} \ | \ 0\leq j \leq i]}$ (where $i\in \Z_{\geq 0}$) 
		such that $f(\widehat {\mathbf x}+p^l\cdot \widehat {\mathbf z})$ (for arbitrary $\widehat {\mathbf z}\in \EE^n$) is given by the Witt vector whose $i$-th component is
		the evaluation of $f_i$ at $\mathbf x$ and $\rho_{i+1}(\widehat {z}_1), \ldots, \rho_{i+1}(\widehat {z}_n)$.
		The polynomials $f_i$ do not depend on $k'$.

		 Now $\nu_{p,\mathbf x}(f) \geq r$ if and only if, for all 
		$0\leq i < r$, $\mathbf x$ is a zero of all coefficients of $f_i$ as a polynomial in $ k[X_{1,0},\ldots, X_{1,l-1}, \ldots, X_{n,0},\ldots, X_{n,l-1}][Z_{1,j},\ldots, Z_{n,j} \ | \ 0\leq j \leq i]$. Hence we can define $\mathcal V_{l,\nu_{p,-}(f)\geq r}$ as the zero locus of these coefficients, which are elements of $ k[X_{1,0},\ldots, X_{1,l-1}, \ldots, X_{n,0},\ldots, X_{n,l-1}]$.
	\end{proof}


	\section{Smooth families of lattices and rigidity}

	In this section we will prove the main ingredient going into Theorem~A, which is Theorem~\ref{thm smooth family open subset} below. To do this, we first need to introduce the structure we use to endow parametric families of lattices with the structure of a variety over $k$. We call this structure a \emph{smooth family of $\Lambda$-lattices}. Its definition is  quite straight-forward, and it is not difficult to show that we can parametrise all $\Lambda$-lattices of a given index in a fixed lattice by finitely many such families (see \S\ref{section lattice var}). Note that lattices in a smooth family as defined below have isomorphic $K$-span, and therefore the same character in case $\Lambda$ is a block or a group algebra.

	\begin{defi}\label{defi smooth family}
		A \emph{smooth family of $\Lambda$-lattices} $L(-)$ is a pair $(\Delta, \mathcal Z \cap \mathcal U)$, where 
		\begin{enumerate}
			\item $\Delta: \Lambda \longrightarrow K[X_1, \ldots, X_n]^{m\times m}$ (for certain $m,n\in \N$) is  a representation,
			\item $\mathcal Z \cap \mathcal U$ is the intersection of  an irreducible closed subvariety $\mathcal Z \subseteq \mathbb A^{n\cdot l}=\mathbb A^l\times \ldots \times \mathbb A^l$, and open subvariety $\mathcal U \subseteq  \mathbb A^{n\cdot l}=\mathbb A^l\times \ldots \times \mathbb A^l$ (for some $l\in \N$), both defined over $k$,
		\end{enumerate}
		such that the following hold:
		\begin{enumerate}
		\item ${\rm sing}(\mathcal Z)\cap \mathcal U = \emptyset$.
		\item Let $\EE=W(k')$ for some algebraically closed field $k'\supseteq k$.
		If $\widehat{\mathbf x} \in \EE^n$ reduces to a point $\mathbf x = \rho_l(\mathbf{\widehat x}) \in \mathcal Z(k') \cap \mathcal U(k')$
		then
		\begin{equation}
			\Delta_{\widehat {\mathbf x}}:\ \Lambda \longrightarrow \EE^{m\times m}:\ \lambda \mapsto \Delta(\lambda)|_{(X_1,\ldots, X_n)=(\widehat x_1, \ldots, \widehat x_n)}
		\end{equation}
		is well-defined, and the isomorphism type of the corresponding $\EE\otimes_{\OO}\Lambda$-lattice only depends on $\mathbf x$.
		\end{enumerate}
	\end{defi}

	A few remarks are in order. Firstly, the elements of $\mathbb A^{n\cdot l}(k')$ represent the first $l$ Witt vector components of elements in $\EE^n$. To reflect this fact we identify the coordinate ring of $\mathbb A^{n\cdot l}$ with $k[X_{i,j}\ | \ 1\leq i \leq n,\ 0\leq j \leq l-1]$. A second thing to note is that the reason we are considering extensions $\EE$ of $\OO$ is that we need those to specialise at ``generic points'' in the proof of Theorem~\ref{thm smooth family open subset} below. In applications we actually only require the property that specialisation at points in $\mathcal Z(k)\cap \mathcal U(k)$ is well-defined. The last thing we should note is that the requirement that $\mathcal Z\cap \mathcal U$ be smooth and irreducible only serves to get uniqueness in Corollary~\ref{cor rigid unique} and make the proof of Theorem~\ref{thm smooth family open subset} slightly nicer. As such, its inclusion in the definition is really a matter of preference.	

	\begin{notation}
		In the situation of Definition~\ref{defi smooth family} we denote the $\EE\otimes_\OO\Lambda$-lattice  corresponding to a point $\mathbf x \in \mathcal Z(k')\cap \mathcal U(k')$ by $L({\mathbf x})$. Of course $L({\mathbf x})$ is only defined up to isomorphism.
		We say that $L(-)=(\Delta, \mathcal Z\cap \mathcal U)$ \emph{contains} (the isomorphism class of) $L({\mathbf x})$. We also refer to the (common) $\OO$-rank of $\Lambda$-lattices contained in $L(-)$ as the ``$\OO$-rank of $L(-)$''. 
	\end{notation}

	\begin{thm}\label{thm smooth family open subset}
		If $L(-)=(\Delta, \mathcal Z\cap \mathcal U)$ is a smooth family of $\Lambda$-lattices, then the subset of ${\mathcal U(k) \cap \mathcal Z(k)}$ such that the corresponding $\Lambda$-lattices are isomorphic to some fixed rigid $\Lambda$-lattice contains a Zariski-open subset.
	\end{thm}
	\begin{proof}
		Assume that we have an $\mathbf x \in \mathcal Z(k)\cap \mathcal U(k)$ such that $L({\mathbf x})$ is rigid.
		Let $\mathfrak m$ be the ideal of elements of $k[\mathcal Z]$ vanishing at $\mathbf x$. As $\mathcal Z$ is smooth at $\mathbf x$, the completion of the local ring $k[\mathcal Z]_{\mathfrak m}$ is isomorphic to 
		$k[[T_1,\ldots, T_{\dim \mathcal Z}]]$. Hence we get a map ${\varphi:\ k[X_{1,0},\ldots,X_{1,l-1},\ldots, X_{n,0},\ldots, X_{n,l-1}]\longrightarrow k[[T_1,\ldots, T_r]]}$, where $r=\dim (\mathcal Z)+n$, with image contained in $k[[T_1,\ldots,T_{\dim\mathcal Z}]]$,  which factors through $k[\mathcal Z]$, and whose composition with the evaluation map $k[[T_1,\ldots,T_r]]\twoheadrightarrow k$ is the same as evaluation at $\mathbf x$ (the $n$ spare variables will come in handy later on). 
		Define ${\mathbf w = (\varphi(X_{i,j}) \ |\ 1\leq i\leq n, 0 \leq j\leq l-1)\in k[[T_1,\ldots,T_r]]^{n\cdot l}}$.
		Then $\mathbf w$ lies in $\mathcal Z(\overline{k((T_1,\ldots, T_r))})$, and $\mathbf w(0,\ldots,0)=\mathbf x$. Here $\mathbf w(0,\ldots,0)$ denotes the evaluation of $\mathbf w$ at $(T_1,\ldots,T_r)=(0,\ldots,0)$. 
		
		Define
		\begin{equation}
			R=\bigcup_{i=0}^\infty \OO[[T_1^{1/p^{i}},\ldots,T_r^{1/p^{i}}]] \quad\textrm{and}\quad \EE_0=\varprojlim R/p^i R.
		\end{equation}
		That is, $\EE_0$ is the $p$-adic completion of $R$. The residue ring $\EE_0/p\EE=R/pR$ can be identified with
		\begin{equation}
			  \bar R=\bigcup_{i=0}^\infty k[[T_1^{1/p^{i}},\ldots,T_r^{1/p^{i}}]],
		\end{equation}
		which is a perfect ring of characteristic $p$.  In particular $\EE_0$  is a strict $p$-ring. Hence $\EE_0$ is isomorphic to $W(\bar R)$ by a unique isomorphism preserving the surjection onto $\bar R$. Now $W(\bar R)$ embeds into  $\EE=	 W(k')$, where 
		${k'=\overline{k((T_1,\ldots,T_r))}}$. Long story short, we get an embedding of $\OO$-algebras $R \hookrightarrow \EE = W(k')$ factoring through an embedding of rings of Witt vectors $\EE_0\hookrightarrow \EE$.
		
		Now choose, for each $1\leq i\leq n$, a $\widehat w_i'\in R$ such that $\rho_l(\widehat w'_i)=(w_{i,0},\ldots, w_{i,l-1})$. 
		The reason we can do this is that $\rho_l:\ \EE_0 \longrightarrow \bar R ^l$ factors through $\EE_0/p^l\EE_0$, and $R$ surjects onto $R/p^lR \cong \EE_0/p^l\EE_0$, the latter two being canonically isomorphic since $\EE_0$ is the completion of $R$. Thus we get a $\widehat{\mathbf w}'=(\widehat w'_1,\ldots,\widehat w'_n)\in R^n$ such that 
		$\rho_l(\mathbf{\widehat w'})=\mathbf w$. Since $R$ is defined as a union of rings, we actually get that 
		\begin{equation}
		\mathbf{\widehat w'}\in \OO[[S_1,\ldots,S_r]]^n\textrm{ where $S_1=T_1^{1/p^{e}},\ldots, S_r=T_r^{1/p^{e}}$ for some $e\in \N$.} 
		\end{equation}
		Since $\mathbf w$ only involves the indeterminates $T_1,\ldots,T_{\dim \mathcal Z}$, we can actually assume that ${\mathbf{\widehat w'}\in \OO[[S_1,\ldots,S_{\dim \mathcal Z}]]^n}$. We then define 
		$\widehat w_i = \widehat w_i' + p^l\cdot S_{\dim \mathcal Z+i}$ for $1\leq i \leq n$. The upshot is 
		that $\mathbf{\widehat w}=(\widehat w_1,\ldots,\widehat w_n)\in \OO[[S_1,\ldots,S_r]]^n$ satisfies $\rho_l(\mathbf{\widehat w})=\mathbf w$ and the individual entries $\widehat w_i$ are algebraically independent over $K$. Moreover, $\rho_l(\mathbf{\widehat w}(0,\ldots,0))=\mathbf x$, since $\rho_l$ commutes with evaluation at $(0,\ldots,0)$ by the universal property of Witt vectors. Note that our construction of $\mathbf{\widehat w}$ also ensures that $\nu_{p,\mathbf w}(g)= \nu_p(g(\widehat{\mathbf w}))$ for all $g\in \OO[X_1,\ldots,X_n]$ (where we view $\widehat{\mathbf w}$ as an element of $\EE^n$), since $\mathbf{\widehat w}$ is polynomial in the spare variables $T_{\dim \mathcal Z+i}$, and therefore 
		\begin{equation}
		\nu_{p,\mathbf w}(g)=\nu_p(g(\widehat{w}'_1 + p^l\cdot T_{\dim\mathcal Z+1}, \ldots, \widehat{w}'_n + p^l\cdot T_{\dim\mathcal Z+n}))=\nu_p(g(\widehat{\mathbf w})),
		\end{equation}
		which is seen by
		using Proposition~\ref{prop generic val first properties} and the fact that $\nu_p$ is independent of whether we regard the $T_{\dim\mathcal Z+i}$ as indeterminates or as elements of the valuation ring $\EE$.

		 Using $\widehat{\mathbf w}$ we can now define an $\OO$-algebra homomorphism 
		 \begin{equation}
		 \widehat \varphi:\ \OO[X_1,\ldots, X_n]\longrightarrow \OO[[S_1,\ldots, S_r]]: X_i \mapsto \widehat w_i
		 \end{equation} 
		 ``lifting'' $\varphi$. Our assumptions ensure  that $\widehat \varphi$ is injective, which will be important later.
		 Moreover, the element $\widehat{\mathbf{w}}$, by construction, gives us representations $\Delta_{\widehat{\mathbf w}}$ 
		 and  $\Delta_{\widehat{\mathbf w}(0,\ldots,0)}$, affording $L({\mathbf w})$ and $L({\mathbf x})$, respectively. 	
		 A priori, the image of $\Delta_{\widehat{\mathbf w}}$ is only contained in $m\times m$-matrices over 
		 $\EE$. However, since the images of $\Delta_{\widehat{\mathbf w}}$ are obtained from the images of $\Delta$ (which live in $K[X_1,\ldots,X_n]^{m\times m}$) by substituting $X_i=\widehat w_i$, we actually get that the image of $\Delta_{\mathbf{\widehat w}}$ is also contained in $m\times m$-matrices over $K\cdot \OO[[S_1,\ldots, S_r]]$. Hence the images of $\Delta_{\mathbf{\widehat w}}$ actually have entries in $\EE \cap (K\cdot \OO [[S_1,\ldots, S_r]])=\OO [[S_1,\ldots,S_r]]$. 
		 
		 It now follows that
		 \begin{equation}\label{eqn expansion delta}
		 	\Delta_{\widehat{\mathbf w}}(\lambda) = \Delta_{\widehat{\mathbf w}(0,\ldots,0)}(\lambda)+\sum_{\mathbf i\in \Z_{\geq 0}^r\setminus \{(0,\ldots,0)\}} \Gamma_{\mathbf i}(\lambda) \cdot S_1^{i_1}\cdots S_r^{i_r}\quad \textrm{for all $\lambda \in \Lambda$,}
		 \end{equation}
		 where each $\Gamma_{\mathbf i}$ is a map from $\Lambda$ into $\OO^{m\times m}$. Assume that at least one of the $\Gamma_{\mathbf i}$ is not the zero map, and  let $\mathbf j\in \Z_{\geq 0}^r$ be degree-lexicographically minimal such that $\Gamma_{\mathbf j}$ is non-zero. Then 
		 \begin{equation}
		 	\Gamma_{\mathbf j}(\lambda\cdot \gamma)= \Delta_{\widehat{\mathbf w}(0,\ldots,0)}(\lambda)\cdot \Gamma_{\mathbf j}(\gamma) + 
		 	\Gamma_{\mathbf j}(\lambda)\cdot  \Delta_{\widehat{\mathbf w}(0,\ldots,0)}(\gamma)
		 \end{equation}
		 for any $\lambda,\gamma \in \Lambda$. That is, $\Gamma_{\mathbf j}$ defines an element of 
		 $\Ext^1_{\Lambda}(L({\mathbf x}), L({\mathbf x}))$, which we assumed was zero. Hence there is a $\mathbf B \in \OO^{m\times m}$ such that, for any $\lambda \in \Lambda$, $\Gamma_{\mathbf j}(\lambda)= \Delta_{\widehat{\mathbf w}(0,\ldots,0)}(\lambda) \cdot \mathbf B - \mathbf B \cdot \Delta_{\widehat{\mathbf w}(0,\ldots,0)}(\lambda)$. Conjugating $\Delta_{\widehat{\mathbf w}}$ by $\id + \mathbf B\cdot S_1^{j_1}\cdots S_r^{j_r}$ yields a representation with an expansion as in  \eqref{eqn expansion delta} where the degree-lexicographically smallest $\mathbf i$ with $\Gamma_{\mathbf i}\neq 0$ is strictly bigger than $\mathbf j$. Iterating this process gives a sequence of conjugating elements in $\GL _m(\OO[[S_1,\ldots S_r]])$ that converge with respect to the topology induced by the realisation of $\OO[[S_1,\ldots, S_r]]$ as $\varprojlim \OO[S_1,\ldots, S_r]/(S_1,\ldots,S_r)^i$. The limit of these elements will conjugate $\Delta_{\widehat{\mathbf w}}$ to $\Delta_{\widehat{\mathbf w}(0,\ldots,0)}$.
		 
		 We now know that $\EE\otimes_{\OO} L({\mathbf x})$ is isomorphic to $L({\mathbf w})$, as these are the modules afforded by the two representations we just showed are conjugate. In elementary terms, this means that the linear system of equations 
		 \begin{equation}\label{eqn hom lats}
		 	\Delta_{\widehat{\mathbf{w}}}(\lambda) \cdot \mathbf M - \mathbf M\cdot \Delta_{\widehat{\mathbf{w}}(0,\ldots, 0)}(\lambda) \quad \textrm{(for $\lambda$ in a basis of $\Lambda$)}
		 \end{equation}
		 has a solution $\mathbf M\in \GL_m(\EE)$. However, the coefficients of the system of equations 
		 \eqref{eqn hom lats} lie in the subring of $\EE$ generated by elements of the discrete valuation ring $R_{pR}$ which have  preimages in $K(X_1,\ldots, X_n)$ under the map $\widehat \varphi$ from earlier (extended to fields of fractions). Hence the $\EE$-lattice of solutions to  \eqref{eqn hom lats} has a basis consisting of matrices $\mathbf M_1,\ldots,\mathbf M_d$ (with $d=\rank_\EE \Hom(L({\mathbf w}), \EE\otimes_{\OO}L({\mathbf x}))$) with entries in the aforementioned subring.  Now the fact that 
		 $L_{\mathbf w}$ and  $\EE\otimes_{\OO}L_{\mathbf x}$ are isomorphic is equivalent to the assertion that the polynomial
		 $$
		 	\det(\mathbf M_1 \cdot Z_1 + \ldots + \mathbf M_d \cdot Z_d) \in \EE[Z_1, \ldots, Z_d]
		 $$
		 has $p$-valuation zero (i.e. some coefficient has $p$-valuation zero). But then one can easily find $\mathbf z\in \OO^d$ (or even $W(\bar{\mathbb F}_p)^d$) such that the $p$-valuation of $\det(\mathbf M_1\cdot z_1+\ldots+\mathbf M_d\cdot z_d)$ is zero. Then $\mathbf M=\mathbf M_1\cdot z_1+\ldots+ \mathbf M_d\cdot z_d$ is an element of $\GL_m(\EE)$ with entries that have preimages in $K(X_1,\ldots, X_n)$ such that $\mathbf M^{-1} \cdot \Delta_{\widehat{\mathbf w}}(\lambda) \cdot \mathbf M = \Delta_{\widehat{\mathbf w}(0,\ldots,0)}(\lambda)$ for all $\lambda \in \Lambda$. Now recall that $\Delta_{\widehat{\mathbf w}}(\lambda)$ is obtained from $\Delta(\lambda)$ by entry-wise application of $\widehat \varphi$ (again, extended to fields of fractions). Let $\mathbf M'$ be a preimage under $\widehat \varphi$ of $\mathbf M$, that is, $\mathbf M'$ has entries in $K(X_1,\ldots,X_n)$. Then $\mathbf M'^{-1}\cdot \Delta(\lambda) \cdot \mathbf  M'$ must be a preimage under $\widehat \varphi$ of $\Delta_{\widehat{\mathbf w}(0,\ldots,0)}(\lambda)$ (for any $\lambda \in \Lambda$). But since $\widehat \varphi$ is injective by construction, we get 
		 \begin{equation}\label{eqn conj}
		 	\mathbf M'^{-1}\cdot \Delta(\lambda) \cdot \mathbf M' = \Delta_{\widehat{\mathbf{w}}(0,\ldots,0)}(\lambda) \quad \textrm{for all $\lambda \in \Lambda$,} 
		 \end{equation}  
		 which is now an equation entirely in $K(X_1,\ldots, X_n)={\rm frac}(\OO[X_1,\ldots, X_n])$. 
		 
		 Now let $\mathbf y$ be another point in $\mathcal U(k)\cap \mathcal Z(k)$, and let $\widehat{\mathbf y}\in \OO^n$ be an element such that $\rho_l(\widehat{\mathbf y})=\mathbf y$. Since $\Delta_{\widehat{\mathbf y}}$ is obtained from $\Delta$ by substituting $X_i=\widehat y_i$, equation~\eqref{eqn conj} implies that $L({\mathbf y})\cong L({\mathbf x})$ provided $\mathbf M'|_{(X_1,\ldots,X_n)=\widehat {\mathbf y}}\in \GL_m(\OO)$. Note that all $g\in \OO[X_1,\ldots,X_n]$ which occur as numerators or denominators of either $\mathbf M'$ or $\mathbf M'^{-1}$ satisfy $\nu_{p,\mathbf w}(g)=\nu_p(g(\widehat{\mathbf w}))$, and substituting $(X_1,\ldots,X_n)=\widehat{\mathbf w}$ in $\mathbf M'$ and $\mathbf M'^{-1}$ gives back $\mathbf M$ and $\mathbf M^{-1}$ by definition. By Proposition~\ref{prop gen val zariski closed} there are closed subvarieties $\mathcal V_{l,\nu_{p,-}(g)\geq\nu_{p,\mathbf w}(g)}$ and $\mathcal V_{l,\nu_{p,-}(g)\geq\nu_{p,\mathbf w}(g)+1}$ of $\mathbb A^{n\cdot l}$ defined over $k$ such that $\mathbf w$ lies in $\mathcal V_{l,\nu_{p,-}(g)\geq\nu_{p,\mathbf w}(g)}(k')$ but not in $\mathcal V_{l,\nu_{p,-}(g)\geq\nu_{p,\mathbf w}(g)+1}(k')$. Since $\mathbf w$ was chosen as a generic point for $\mathcal Z$, any subvariety of $\mathbb A^{n\cdot l}$ defined over $k$  contains $\mathbf w$ if and only if it contains $\mathcal Z$. It follows that $(\mathcal V_{l,\nu_{p,-}(g)\geq\nu_{p,\mathbf w}(g)} \setminus \mathcal V_{l,\nu_{p,-}(g)\geq\nu_{p,\mathbf w}(g)+1})\cap \mathcal Z$ is an open subvariety of $\mathcal Z$, whose $k$-rational points are by definition those $\mathbf y$ for which $\nu_{p, \mathbf y}(g)=\nu_{p, \mathbf w}(g)$.
		 
		  We conclude that there is an open subvariety $\mathcal V$ of $\mathcal U$  such that for any $\mathbf y\in \mathcal V(k)$ there is a $\widehat {\mathbf y}$ for which $\mathbf M'|_{(X_1,\ldots,X_n)=\widehat {\mathbf y}}$
		 and $\mathbf M'^{-1}|_{(X_1,\ldots,X_n)=\widehat {\mathbf y}}$ lie in $\OO^{m\times m}$ (in fact, the valuation of each entry is the same as that of the corresponding one of $\mathbf M$). Hence $L({\mathbf y})\cong L({\mathbf x})$ for all $\mathbf y \in \mathcal V(k)$, which completes the proof.
	\end{proof}

	\begin{corollary}\label{cor rigid unique}
		A smooth family of $\Lambda$-lattices contains at most one isomorphism class of rigid lattices. 
	\end{corollary}
	\begin{proof}
		We assume $\mathcal Z$ to be irreducible, which means that any two non-empty Zariski-open subsets have non-trivial intersection. In particular, if the family contains two rigid lattices, then their respective sets of points parametrising lattices  isomorphic to them have non-trivial intersection. This implies that any two rigid lattices in the family must be isomorphic.
	\end{proof}

	\section{Varieties of lattices}\label{section lattice var}

	In this section we will show how to parametrise all $\Lambda$-lattices in a given finite-dimensional $K\otimes_\OO\Lambda$-module $V$, up to isomorphism. The idea is to start with a fixed lattice, and then to conjugate an affording representation by an upper-triangular basis matrix with ``generic entries'' above the diagonal. The condition that the specialisation of the resulting ``generic representation'' $\Delta$ at some $\widehat{\mathbf x}$ be integral is a closed condition on the Witt vector components of $\widehat{\mathbf x}$, which, after a minimal amount of work, gives sets $\mathcal Z \cap \mathcal U$ such that each $(\Delta, \mathcal Z\cap \mathcal U)$ defines a smooth family of $\Lambda$-lattices.  
	
	\begin{lemma}\label{lemma sublattice fams}
		Let $L$ be a $\Lambda$-lattice of rank $m\in \N$ and let $l$ be some non-negative integer. Then there are finitely many smooth families of 
		$\Lambda$-lattices $L_1(-),\ldots,L_d(-)$ ($d\in\N$) such that each $\Lambda$-sublattice $L'\leq L$ for which the quotient $L/L'$ has length $l$ as an $\OO$-module is contained in one of the $L_i(-)$.
	\end{lemma}
	\begin{proof}
	Let $\Delta_L:\ \Lambda \longrightarrow \OO^{m\times m}$ be a representation affording $L$, and let us fix $v_1,\ldots, v_m\in \Z_{\geq 0}$ such that $v_1+\ldots+v_m=l$. Let us also fix an algebraically closed $k'\supseteq k$ and set $\EE=W(k')$ (this is just to formally verify the conditions of a smooth family of $\Lambda$-lattices). Consider the upper-diagonal matrix
	\begin{equation}
	\mathbf B=
	\left( \begin{array}{ccccc}
	p^{v_1} & X_{1} & X_{2} & \cdots & X_{m-1} \\
	0 & p^{v_2} & X_{m}& \cdots& X_{2m-3} \\
	0 & 0 & p^{v_3} & \cdots & X_{3m-6} \\
	\vdots & \vdots  & \vdots & \ddots & \vdots \\ 
	0 & 0 & 0 & \cdots & p^{v_m}\
	\end{array}\right) \in K[X_1,\ldots, X_{m(m-1)}]^{m\times m}
	\end{equation} 
	We will define finitely many smooth families of $\Lambda$-lattices such that every $\Lambda$-sublattice of $\OO^m$ (considered as a $\Lambda$-lattice via $\Delta_L$) which has an $\OO$-basis  consisting of the rows of $\mathbf B|_{(X_1,\ldots,X_{m(m-1)})=\widehat{\mathbf x}}$ for some $\widehat{\mathbf x}\in\OO^{m(m-1)}$ is contained in one of these families. That will actually prove the lemma, since every $\OO$-sublattice of $\OO^m$ with quotient of $\OO$-length $l$ has a basis given by the rows of a matrix of the same form as $\mathbf B$ for some $v_1,\ldots, v_m\in \Z_{\geq 0}$ with $v_1+\ldots+v_m=l$.
	
	First define
	\begin{equation}
	\Delta:\ \Lambda\longrightarrow K[X_1,\ldots,X_{m(m-1)}]^{m\times m}:\ \lambda\mapsto \mathbf B \cdot \Delta_L(\lambda)\cdot \mathbf B^{-1} 
	\end{equation}
	 Note that $p^l\cdot \Delta$ takes values in $\OO[X_1,\ldots,X_{m(m-1)}]^{m\times m}$, and for a given $\lambda\in \Lambda$ and $1\leq i,j\leq m$ we have that
	$p^l \cdot \Delta(\lambda)_{i,j}$ is an element  $f\in {\OO[X_1,\ldots,X_{m(m-1)}]}$. 
	From the theory of Witt vectors we get polynomials $f_0,\ldots, f_{l-1}\in {k[X_{i',j'} \ | \ 1\leq i' \leq m(m-1), 0 \leq j' \leq l-1]}$ such that 
	\begin{equation}
	\rho_l (f(\widehat{\mathbf x}))=\left(f_0(\rho_l(\widehat{\mathbf x})), \ldots, f_{l-1}(\rho_l(\widehat{\mathbf x}))\right)
	\end{equation}
	for all $\widehat{\mathbf x} \in \EE^{m(m-1)}$. 
	The intersection of the vanishing sets of the polynomials $f_0,\ldots, f_{l-1}$ defines a closed subvariety 
	$\mathcal X_{i,j,\lambda} \subseteq \mathbb A^{m(m-1)\cdot l}$ defined over $k$.
	By definition, the representation $\Delta_{\widehat{\mathbf x}}$ (for  $\widehat{\mathbf x} \in \EE^{m(m-1)}$) takes values in $\EE^{m\times m}$ if and only if $\widehat{\mathbf x}$ reduces to a point of $\mathcal X(k')$, where  
	\begin{equation}
		\mathcal X= \bigcap_{i,j=1}^m \bigcap_{\lambda} \mathcal X_{i,j,\lambda}\quad \textrm{ ($\lambda$ running over a basis of $\Lambda$)}
	\end{equation}
	Moreover, one checks that the row space of the matrix $\mathbf B$ specialised at any element of $\EE^{m(m-1)}$ contains $p^l\cdot \EE^m$, which implies that if $\widehat{\mathbf x}, \widehat{\mathbf y}\in \EE^{m(m-1)}$ reduce to the same point in $\mathcal X(k')$, then the respective row spaces of the matrices $\mathbf B|_{(X_1,\ldots,X_{m(m-1)})=\widehat{\mathbf x}}$ and $\mathbf B|_{(X_1,\ldots,X_{m(m-1)})=\widehat{\mathbf y}}$ are equal, which implies that the representations $\Delta_{\mathbf{\widehat x}}$ and $\Delta_{\mathbf{\widehat y}}$ are conjugate. 
	
	The only remaining problem is the fact that $\mathcal X$ is in general neither smooth nor irreducible. However, we can decompose $\mathcal X$ as a union of finitely many irreducible components, remove the singular loci from each irreducible component, then decompose the singular loci into irreducible components, and so on and so forth. We ultimately obtain finitely many $\mathcal Z_1,\ldots, \mathcal Z_d \subseteq \mathbb A^{m(m-1)\cdot l}$ closed 
	and $\mathcal U_1,\ldots, \mathcal U_d \subseteq \mathbb A^{m(m-1)\cdot l}$ open ($d\in \N$) such that
	\begin{equation}
		\mathcal X = \bigcup_{i=1}^d \mathcal Z_i \cap \mathcal U_i
	\end{equation}
	By construction, the $L_i(-)=(\Delta, \mathcal Z_i\cap \mathcal U_i)$ are smooth families of $\Lambda$-lattices such that each $\Lambda$-sublattice of $L$ with a basis of the same shape as $\mathbf B$ is contained in one of the $L_i(-)$, as required.
	\end{proof}
	
	\begin{thm}\label{thm lattice alg}
		Let $V$ be a finite-dimensional $K\otimes_\OO \Lambda$-module. Then there are finitely many smooth families of $\Lambda$-lattices 
		$M_1(-),\ldots, M_d(-)$ such that each full $\Lambda$-lattice $L\leq V$ is contained in one of the $M_i(-)$.
	\end{thm}
	\begin{proof}
		Fix some full $\Lambda$-lattice $L_0\leq V$.
		It is well-known that every full $\Lambda$-lattice in $V$ is isomorphic to a $\Lambda$-lattice $L$ such that $L_0\cdot (\Gamma:\Lambda)^2 \leq L \leq L_0$, where $(\Gamma : \Lambda)$ denotes the biggest two-sided $\Gamma$-ideal contained in $\Lambda$,  $\Gamma$ being a maximal order containing $\Lambda$. Therefore there is an upper bound $n\in \N $ on the composition length of $L_0/L$ as an $\OO$-module which depends only on $V$ and $\Lambda$. Now we can just apply Lemma~\ref{lemma sublattice fams} to $L_0$ for all $0\leq l\leq n$, such as to obtain finitely many smooth families of $\Lambda$-lattices containing all $\Lambda$-sublattices of~$L$, up to isomorphism. 
	\end{proof}

	\section{Proofs of the main theorems and applications}

	It should be fairly clear by now how Corollary~\ref{cor rigid unique} and Theorem~\ref{thm lattice alg} imply Theorem~A, and Theorem~B is an immediate consequence of that. We still include proofs for completeness' sake. 

	\begin{proof}[Proof of Theorem~A]
		As $K\otimes_{\OO} \Lambda$ is assumed to be separable, there are only finitely many isomorphism classes of $K\otimes_\OO\Lambda$-modules $V$ of dimension $\leq n$.
		By Theorem~\ref{thm lattice alg} there are, for each such $V$, finitely many smooth families of $\Lambda$-lattices such that each full $\Lambda$-lattice in $V$ is contained in one of these families. Hence every $\Lambda$-lattice of rank $\leq n$ is contained in one of finitely many smooth families of $\Lambda$-lattices, and by Corollary~\ref{cor rigid unique} each such family can contain at most one isomorphism class of rigid lattices.
	\end{proof}

	\begin{proof}[Proof of Theorem~B]
		We can assume without loss of generality that $\Lambda$ is basic. Then every element of $\Pic_{\OO}(\Lambda)$ is represented by a $\Lambda$-$\Lambda$-bimodule $\Lambda{_\alpha}$, where $\alpha \in \Aut_{\OO}(\Lambda)$.
		Define $\Lambda^{\rm e} = \Lambda^{\rm op}\otimes_{\OO} \Lambda$. Then $\Lambda^{\rm e}$ is again an $\OO$-order in a separable $K$-algebra, and we can view elements of $\Pic_{\OO}(\Lambda)$ as 
		$\Lambda^{\rm e}$-lattices. Note that if $\alpha\in \Aut_\OO(\Lambda)$, then $\id\otimes \alpha\in \Aut_{\OO}(\Lambda^{\rm e})$, and $\Lambda_{\alpha}$ (that is, the $\Lambda$-$\Lambda$-bimodule $\Lambda$ twisted by $\alpha$ on the right) is the same as $\Lambda_{\id\otimes \alpha}$ (that is, the right $\Lambda^{\rm e}$-module $\Lambda$ twisted by $\id\otimes \alpha$). Hence 
		$\Ext^1_{\Lambda^{\rm e}}(\Lambda_\alpha, \Lambda_\alpha)\cong \Ext^1_{\Lambda^{\rm e}}(\Lambda, \Lambda)=\HHH^1(\Lambda)$, and the latter is zero by assumption. That is, the elements of $\Pic_\OO(\Lambda)$ are rigid $\Lambda^{\rm e}$-lattices of $\OO$-rank equal to the $\OO$-rank of $\Lambda$. By Theorem~A there are only finitely many such $\Lambda^{\rm e}$-lattices, up to isomorphism.
	\end{proof}
	Proposition~\ref{prop det pic} is actually just the combination of Theorem~B and Weiss' criterion.
	\begin{proof}[Proof of Proposition~\ref{prop det pic}] 
		Consider $(\OO G b)^{\rm e}$ as a block of $\OO (G^{\rm op}\times G)$.
		Let $M$ be an {$(\OO G b)^{\rm e}$-module} representing an element of $\operatorname{Pic}(\OO Gb)$. By Weiss' criterion (see \cite{WeissRigidity}, and \cite{WeissArbitraryCoeffs} for a version allowing $\OO$ as a coefficient ring) our $M$ has trivial source if and only if the lattice of $P$-fixed points taken on the left
		\begin{equation}		
		{^{P}M} \cong {^P(\OO Gb\otimes _{\OO G b} M)} \cong {(^P\OO Gb)\otimes _{\OO G b} M}   \cong \OO [P\backslash G]\cdot b\otimes_{\OO Gb} M  
		\end{equation} 
		has trivial source as an $\OO Gb$-module, that is, is $p$-permutation.  That proves the first part. 
		
		Now if $M$ represents an element of $\operatorname{Picent}(\OO Gb)$, then  
		$\OO[P\backslash G]\cdot b\otimes_{\OO Gb}M$ has the same character as $\OO[P\backslash G]\cdot b$, and isomorphic endomorphism ring. Moreover, we clearly have 
		\begin{equation}
		\Ext^1_{\OO G b}(\OO[P\backslash G]\cdot b\otimes_{\OO Gb}M, \OO[P\backslash G]\cdot b\otimes_{\OO Gb}M)\cong \Ext^1_{\OO G b}(\OO[P\backslash G]\cdot b, \OO[P\backslash G]\cdot b)=0.
		\end{equation}
		Hence, by our uniqueness assumption, we have 
		$\OO[P\backslash G]\cdot b\otimes_{\OO Gb}M \cong \OO[P\backslash G]\cdot b$ as $\OO Gb$-modules, which, by the above, implies that $M$ lies in $\mathcal T(\OO G b)$.
	\end{proof}

	To finish, let us briefly mention the following nice consequence of Theorems~A~and~B, even though it is implied by \cite[Theorem (38.6)]{Thevenaz}, a theorem due to Puig.
	\begin{prop}
		Let $\{\OO G_ib_i\}_{i\in I}$ (for some index set $I$) be a family of block algebras defined over $\OO$ with fixed defect group $P$, each of which  Morita equivalent to some fixed $\OO$-algebra $A$ by means of some fixed $\OO G_ib_i$-$A$-bimodule $M_i$. If there is a bound, independent of $i$,  on the dimension of $L\otimes_{\OO G_ib_i}M_i$ for indecomposable $p$-permutation $\OO G_ib_i$-modules $L$, then the $\OO G_ib_i$ split into finitely many equivalence classes with respect to splendid Morita equivalence. 
	\end{prop}
	\begin{proof}
		Note that by assumption $P$ is a subgroup of $G_i$ for every $i\in I$. Equivalently, one could also assume that there is a fixed embedding $P\hookrightarrow G_i$ for each $i$, but we are going to take the former point of view. 
		By Theorem~A there are only finitely many rigid $A$-lattices of rank smaller than the given bound on images of indecomposable $p$-permutation modules. Hence $I$ splits up into finitely many sets $I_1,\ldots, I_d$ such that 
		$I=\bigcup_{j=1}^d I_j$ and, for every $i\in I_j$, the $A$-module $\OO[P\backslash G_i]\cdot b_i\otimes_{\OO G_ib_i} M_i$ has the same indecomposable summands as some fixed $A$-module $L_j$. It follows that if $i,i'\in I_j$, then $M_i\otimes_A M_{i'}^\vee$ maps the indecomposable summands of $\OO[P\backslash G_i]\cdot b_i$ to the indecomposable summands of $\OO[P\backslash G_{i'}]\cdot b_{i'}$. One can show using the same argument as in the proof of Proposition~\ref{prop det pic} that  $M_i\otimes_{A}M_{i'}^\vee$ is a $p$-permutation $\OO(G_i^{\rm op}\times G_{i'})$-module. By \cite[7.5.1]{PuigBook} (or, independently, by results of \cite{ScottUnpublished}) this implies that the $M_i\otimes_A M_{i'}^\vee$ for $i,i'$ in a fixed $I_j$ are splendid up to restriction along an automorphism of $P$, of which there are only a finite number. That is, each $I_j$ splits into finitely many subsets such that the blocks parametrised by any one of the subsets are pair-wise source algebra equivalent by means of the $M_i\otimes_{A}M_{i'}^\vee$.
	\end{proof}

	\paragraph*{Acknowledgements} I would like to thank Radha Kessar and Markus Linckelmann for pointing the finiteness problem for Picard groups out to me, and for helpful comments on a first version of this preprint.

	\bibliography{refs}
	\bibliographystyle{alpha}	
\end{document}